\documentclass[preprint,12pt]{elsarticle}

\usepackage{amssymb}
\usepackage{amsmath}
\usepackage{amsthm}
\usepackage{array}
\usepackage{rotating}
\usepackage{vmargin}
\setmarginsrb{2.2cm}{1cm}{2.2cm}{3cm}{1cm}{1cm}{1.5cm}{1.5cm}
\newtheorem{thm}{Theorem}
\newtheorem{cor}{Corollary}
\newtheorem{prop}{Proposition}

\newdefinition{rem}{Remark}
\newdefinition{defi}{Definition}
\newcolumntype{M}[1]{>{\raggedright}m{#1}}

\journal{*****}

\begin{document}

\begin{frontmatter}

\author{Thierry COMBOT\fnref{label2}}
\ead{thierry.combot@u-bourgogne.fr}
\address{IMB, Universi\'e de Bourgogne, 9 avenue Alain Savary, 21078 Dijon Cedex }

\title{Hyperexponential solutions of elliptic difference equations}

\author{}

\address{}

\begin{abstract}
Consider an elliptic curve $\mathcal{C}$ with coefficients in $\mathbb{K}$ with $[\mathbb{K}:\mathbb{Q}]<\infty$ and $\delta \in \mathcal{C}(\mathbb{K})$ a non torsion point. We consider an elliptic difference equation $\sum_{i=0}^l a_i(p) f(p\oplus i.\delta)=0$ with $\oplus$ the elliptic addition law and $a_i$ polynomials on $\mathcal{C}$. We present an algorithm to compute rational solutions, then an intermediary class we call pseudo-rational solutions, and finally hyperexponential solutions, which are functions $f$ such that $f(p\oplus \delta)/f(p)$ is rational over $\mathcal{C}$.
\end{abstract}
\end{frontmatter}

\section{Introduction}

Let us consider an elliptic curve under Weierstrass form $\{(u,v)\in \mathbb{C}^2, v^2=u^3+au+b\}$. This curve can be completed by a point at infinity, noted $O$, and the completed curve will be called $\mathcal{C}$. Remark that this completion point is unique because $\sqrt{u^3+au+b}$ is ramified above $u=\infty$. We can now define on $\mathcal{C}$ an addition law $\oplus$ with the property that $p\oplus q \oplus r=O$ if $p,q,r$ are aligned, and $\ominus p$ is the symmetric with respect of the abscissa of $p$. Any straight line cuts $\mathcal{C}$ at $3$ points (with possibly multiplicity), and thus given any pair of points $p,q$, the point $\ominus r$ and then $r$ can be computed.

A point is called of torsion if an integer multiple of it is $O$, which we will note $N.p=O$ where $N\in\mathbb{N}^*$. Remark in particular that the ramification points of $\mathcal{C}$ are of $2$ torsion. For a given point $\delta \in\mathcal{C}$, the addition of $\delta$ defines an automorphism $\phi_{\delta}$ of $\mathcal{C}$. This automorphism is cyclic if and only if the point $\delta$ is of torsion. In the rest of the article, we will consider a point of $\mathcal{C}$ noted $\delta$ which is not of torsion. In this article we will be interested by difference equations of the form
\begin{equation}\label{eqell}
\sum_{i=0}^l a_i(p) f(p \oplus i.\delta)=0
\end{equation}
where $a_i$ are polynomial functions over $\mathcal{C}$, i.e. elements of $\mathbb{C}[u,v]/(u^3+au+b-v^2)$. These kind of difference equations where considered by Dreyfus, Arreche and al. \cite{dreyfus2015galois}, \cite{arreche2021differential} with applications in mind \cite{dreyfus2019differential}, and can be also written using the standard shift by parametrizing $\mathcal{C}$ with Weierstrass elliptic functions. We are interested in three types of explicit solutions.

\begin{defi}
A solution of \eqref{eqell} is called rational if it is an element of $\hbox{Frac}(\mathbb{C}[u,v]/(u^3+au+b-v^2))$. A solution of\eqref{eqell} is called hyperexponential if it also satisfies an elliptic difference equation of order $1$, i.e. $c_1(p) f(p \oplus \delta)+c_2(p) f(p)=0$ with $c_1,c_2$ polynomials, $(c_1,c_2)\neq (0,0)$. A solution of \eqref{eqell} is called pseudo-rational if it is hyperexponential and admits finitely many poles and roots.
\end{defi}

These definitions have a natural equivalence with the classical case of the shift on $\mathbb{C}$, which gives the classical difference equations, see \cite{hardouin2016galois}. These type of solutions were analysed by \cite{petkovvsek1992hypergeometric} and generalization to systems \cite{abramov2015hypergeometric} with algorithms to compute them. The intermediary class we call pseudo-rational solutions is the equivalent of solutions of the form $\lambda^n F(n),\; F\in\mathbb{C}(n)$ for the shift in $\mathbb{C}$. They indeed are hyperexponential and only have finitely many roots and poles. In the elliptic case, they appear as the main building block for computing hyperexponential solutions.

\begin{thm}\label{thm1}
The algorithms \textsf{RationalSolutions}, \textsf{PseudoRationalSolutions}, \textsf{HyperexponentialSolutions} compute a basis of solutions of the vector space of respectively rational solutions, pseudo-rational solutions and hyperexponential solutions.
\end{thm}

\noindent
The approach will be similar to Petkovšek \cite{petkovvsek1992hypergeometric}, but some differences will appear
\begin{itemize}
\item A fundamental quantity is the dispersion. It corresponds to the integer differences between two roots of $a_la_0$. In the shift case on $\mathbb{C}$, this is just the maximum of integer difference between roots of $a_la_0$. In the elliptic case, this is the difference in a multiple of $\delta$ using the addition law on $\mathcal{C}$. Thus we need to test if a point is a multiple of $\delta$, for which will use the notion of canonical height on elliptic curves.
\item For rational solutions, the poles of a possible solution have to be controlled. On a positive side, the point at infinity $O$ is not special any more as in the classical case of the shift in $\mathbb{C}$, as it is not a fixed point. Thus controlling the poles also gives bounds for the degree. On a negative side, this only gives multiplicity bounds at each point, and not a universal denominator as on $\mathbb{C}$, because all divisors are not principal. This implies that a universal denominator could overshoot the possible pole multiplicities.
\item In the classical case of the shift over $\mathbb{C}$, pseudo-rational solutions differ from rational solutions only by a $\lambda^n$ factor. On elliptic curves, not only such possible factors cannot be recovered a priori using local analysis, but they are in fact not enough. All divisors are not principal, and thus cannot be realized with a rational function. However it appears that they can always be realized by a hyperexponential function.
\item Hyperexponential functions which are not pseudo-rational have infinitely many poles or roots. They are located on finitely many orbits. As in the shift case on $\mathbb{C}$, we can bound the possible order of these poles/roots and then a combinatorial analysis begins: for each possible combination of multiplicities, we divide by a hyperexponential function removing these poles and finitely many subsist. However not all combinations of multiplicities can be removed as a principality condition on a divisor appears. This additional constraint leads to a definition of $\delta$-principality, which is necessary for such combination to be realized. As most divisors are not $\delta$-principal, probably only the ones leading to a hyperexponential will be $\delta$-principal, and thus this combinatorial part can be done faster than in the classical case.
\end{itemize}

Before even beginning, we need to ensure that $\delta$ is not a torsion element. The algorithm \textsf{TorsionOrder} of Combot \cite{combot2021elementary} gives an integer $N$ which is a candidate for the torsion order of $\delta$. If it is zero, then $\delta$ is not of torsion. Else if $\delta$ is of torsion, then its order is $N$. This has then to be checked by performing the computation $N.\delta$. Thus it is possible to decide if the equation is well posed. In the article, for computing on examples, we will always choose the elliptic curve $v^2=u^3+15$ with the shift $\delta=(1,4)$. This curve has nice symmetries and a Mordell Weil group over $\mathbb{Q}$ of rank $2$, and thus has many rational points which are not multiple of $\delta$.

\section{Rational solutions}

The first step is to compute the set of possible singularities for a rational solution. There are finitely many orbit under the action of $\phi_\delta$ which can encounter a zero of $a_la_0$. We however need to know if two roots are on the same orbit. For this we will need to bound a priori the distance between two roots on an orbit of $\phi_\delta$.

\begin{defi}[See Silverman \cite{silverman2009arithmetic}]
The naive height of an algebraic point $(\alpha,\beta)\neq O$ of $\mathcal{C}$ is
$$h(\alpha,\beta)=\frac{1}{d} \sum_{\underset{ \hbox{of }\mathbb{Q}(\alpha)}{ v \hbox{ valuation}}} \ln \max(1,\mid \alpha \mid_v )$$
with $d$ the degree of the minimal polynomial of $\alpha$ over $\mathbb{Q}$, and $h(O)=0$. The canonical height is defined by
$$\hat{h}(p)=\lim_{n\rightarrow \infty} \frac{h(2^n.p)}{4^n}$$
This limit exists and defines $\hat{h}$ which has the following properties
\begin{itemize}
\item $\hat{h}(m.p)=m^2 \hat{h}(p)$
\item $\hat{h}(p\oplus q)+ \hat{h}(p\ominus q)= 2\hat{h}(p)+2\hat{h}(q)$
\item $\hat{h}(p)=0 \Leftrightarrow p$ is of torsion
\item $<p,q>=\tfrac{1}{2}(\hat{h}(p\oplus q)-\hat{h}(p)-\hat{h}(q))$ is a positive symmetric bilinear form.
\end{itemize}
\end{defi}

The definition of $\hat{h}$ is not suitable for its practical computation. This is done through computation of local height functions at many primes. The canonical height of a point is (probably) typically a transcendental number, and can only be computed approximatively. However, this is not a problem for our application, where we need to test if a given point is a multiple of $\delta$.

\begin{prop}\label{propideals}
Let $\mathcal{J}_1,\mathcal{J}_2$ be two different maximal ideals of $\mathbb{K}[u,v]$ containing the equation of $\mathcal{C}$. Then two roots of $\mathcal{J}_1$ cannot be on the same orbit of $\phi_\delta$. If a root of $\mathcal{J}_1$ and a root of $\mathcal{J}_2$ are on the same orbit of $\phi_\delta$, then every root of $\mathcal{J}_1$ has a root of $\mathcal{J}_2$ in its orbit.
\end{prop}

\begin{proof}
Assume $\mathcal{J}_1$ has two roots on the same orbit, which we can note $\alpha_1,\alpha_2$. Then $\alpha_2=\alpha_1 \oplus N.\delta$. Now the Galois action is transitive as $\mathcal{J}_1$ is maximal. Thus the Galois group sends $\alpha_1$ to $\alpha_2$, and $\alpha_2$ to $\alpha_i$. If $i=1$ for all $\sigma\in \hbox{Gal}(\mathcal{J}_1)$, then $\{\alpha_1,\alpha_2\}$ is a stable subset by Galois action, and thus should be the zero set of a maximal ideal containing $\mathcal{J}_1$, which is impossible. Thus an index $i\neq 1$ can be obtained, that up to permutation we call $\alpha_3$. We have again $\alpha_3=\alpha_2 \oplus N.\delta$ as $\delta\in \mathcal{C}(\mathbb{K})$.
We now iterate the argument, building on the orbit $\alpha_4$ and so on until all roots of $\mathcal{J}_1$ are on the same orbit. As the elliptic addition is a rational transformation, this implies that all roots have rational expressions in $\alpha_1$, and thus $\sharp \hbox{Gal}(\mathcal{J}_1)=\sharp \mathcal{J}_1^{-1}(0)$. As the Galois group should be transitive, it is then cyclic. But then for the last root $\alpha_n$, we should have $\alpha_1=\alpha_n \oplus N.\delta$, and thus $(nN).\delta=O$ which is forbidden as $\delta$ is not of torsion.
If we take two roots $\alpha,\beta$ of respectively $\mathcal{J}_1,\mathcal{J}_2$ which are on the same orbit, then $\beta$ can be expressed rationally in function of $\alpha$ and vice versa. Thus the field extension defined by $\mathcal{J}_1,\mathcal{J}_2$ are the same, and  the Galois action on $\alpha$ of this field extension span all possible roots of $\mathcal{J}_1$ and thus its action on $\beta$ spans all roots of $\mathcal{J}_2$ too.
\end{proof}

This implies that the singularities of equation \eqref{eqell} can be studied up to Galois conjugacy: we never need to consider two different Galois conjugated roots, and if two roots are on the same orbit, then their conjugate also are. Remark moreover that if two roots are in a field extension $\mathbb{L}$ over $\mathbb{K}$, the two extension have to be the same as any field extension of $\mathbb{K}$ is stable by $\phi_\delta$. This allows to reduce the number of tests to do.

\begin{defi}
A divisor $D$ on $\mathcal{C}$ is a function $D:\mathcal{C} \rightarrow \mathbb{Z}$ which is non zero at finitely many points. It is called proper if its value at point $O$ is
$$D(O)=-\sum_{p\in \mathcal{C}\setminus \{O\}} D(p)$$
A divisor is said to be principal if there exists a rational function $f$ on $\mathcal{C}$ such that $D(p)=\hbox{val}_p(f),\; \forall p\in \mathcal{C}$. Thus the degree of a polynomial is minus the value of its divisor at $O$.
\end{defi}

Principal divisor are proper, but all proper divisors are not principal. Recall that a proper divisor is principal if and only if the sum of its points with multiplicities using the elliptic addition law equals $O$. For example, removing a root of the set of roots of a polynomial can define a divisor which is not principal, and thus is not the set of roots (with multiplicities) of another polynomial. Because of this, it is possible that all coefficients of equation \eqref{eqell} have a common root, but that we cannot simplify it nevertheless. Let us present then an algorithm to compute the roots of $a_0,a_l$ which are not common roots of all coefficients of equation \eqref{eqell}.\\

\noindent
\textsf{Singularties} Input: An equation of the form \eqref{eqell}\\
Output: A list of maximal ideals representing roots and their multiplicities.\\
\begin{enumerate}
\item Compute the prime decomposition $\mathcal{J}_1,\dots,\mathcal{J}_r$ of the ideal generated by $a_l$ and $\mathcal{C}$.
\item For $i=1\dots r$, compute the maximal power of $\mathcal{J}_i$ dividing all coefficients of \eqref{eqell} and the maximal power of $\mathcal{J}_i$ dividing $a_l$. The difference is the multiplicity of $\mathcal{J}_i$.
\item Compute the same for the ideal generated by $a_0$ and $\mathcal{C}$.
\item Return the list of the $\mathcal{J}_i$ with non zero multiplicities counted negatively for $a_l$, and positively for $a_0$.
\end{enumerate}

Remark that we cannot assume that $a_0$ and $a_l$ have distinct roots, so the same $\mathcal{J}_i$ can appear both positively and negatively. We now compute what we call the universal divisor.\\

\noindent
\textsf{UniversalDivisor} Input: An equation of the form \eqref{eqell}\\
Output: A divisor on $\mathcal{C}$
\begin{enumerate}
\item Compute $S=$\textsf{Singularties} of \eqref{eqell}.
\item For each pair of ideals of $S$, check if the field they define are the same. If yes, add the difference of two roots in a set $\Sigma$
\item Compute $b=\left\lceil\hat{h}(\delta)^{-1/2}\max_{p\in\Sigma} \hat{h}(p)^{1/2}\right\rceil$
\item For each element $\mathcal{J}$ of $S$ do
\begin{enumerate}
\item Compute $\phi_\delta^{-i}$ of a root of $\mathcal{J}$ for $i=1\dots b$, and check if one of them is a root of one of the ideals of $S$ other than $\mathcal{J}$.
\item If no, compute $\phi_\delta^i$ of a root $p$ of $\mathcal{J}$ for $i=1\dots b$ and compute a list $L$ of the sum of positive and minus the sum of negative multiplicities for the ideals $\mathcal{J}_i$ vanishing on a point $\phi_\delta^j(p),\; j=0\dots i $.
\item Compute $\min(L_{i-l,1},L_{i,2}),\; i=1\dots b$ with $0$ for non positive indices of $L$.
\end{enumerate}
\item Return the list of $\mathcal{J}$ who passed the first test with their associated list.
\end{enumerate}

The output defines a divisor whose support are the roots of the $\mathcal{J}$ and their shifts, and the value of the divisor is the value of the associated list.

\begin{prop}\label{propuniversal}
The output divisor $D$ of \textsf{UniversalDivisor} is such that for any rational solution of equation \eqref{eqell}, its divisor $D'$ is such that $D'\geq -D$.
\end{prop}

\begin{proof}
Let us consider a rational solution $f$ of equation \eqref{eqell}. Along an orbit of $\phi_\delta$, the equation \eqref{eqell} define the next value in function of the former ones except when the dominant coefficient vanishes. Thus $f$ can only have poles on orbits of $\phi_\delta$ containing a root of $a_l$. However, if a single singularity is met along the orbit, infinitely many singularities would appear. Thus on the same orbit, they have to be compensated further along the orbit by a vanishing trailer coefficient.

In step $1$ the roots of $a_l,a_0$ are computed with multiplicity. In steps $2,3$ a bound is build on the shift in $\delta$ they could have between them: two points on the same orbits are in the same field, and the multiplicative property of $\hat{h}$ implies that if $p\ominus q=N.\delta$, then $\hat{h}(p\ominus q)/\hat{h}(\delta)=N^2$. Remark that these canonical height ratios only need to be computed with precision $1$, and as $\hat{h}(\delta)\neq 0$, only finite precision is needed.

In step $4$, we consider one by one the roots ideals. In step $4a$, we check if the root chosen is the first one on the orbit. If not, then the first one will be encountered further (or already analysed before) in the loop. If the root chosen is the first on its orbit, then we build a list of couple of integers. The first one, is an increasing sequence, which increases each time we encounter a new root of $a_l$. The second one, is an increasing sequence, which increases each time we encounter a new root of $a_0$.

Now for our rational function $f$, we have finitely many poles, and thus in particular finitely many on this orbit. Thus any pole has to be compensated at some point with a root of $a_0$. The highest possible pole multiplicity along the orbit is then given by $\min(L_{i-l,1},L_{i,2}),\; i=1\dots b$, and $0$ after and before. The returned expression in step $5$ is then a list of maximal ideals encoding for roots, and a list of integers giving the maximal order of a pole of $f$ along the orbit by $\phi_\delta$ issued from this root. This property is valid for every point, including $O$, and thus the divisor $D'$ of $f$ is such that $D'\geq -D$.
\end{proof}

This minoration of the divisor of a rational solution allows to bound the degree (which is the valuation at the point $O$), and the order and position of its poles. Remark that even the affine part of the divisor $D$, completed at $O$ such that it becomes proper, is not always principal. Thus this set of poles cannot always be represented by a single polynomial.

\begin{prop}\label{propreduc}
Every proper divisor can be reduced modulo the principal divisors to a divisor of the form $\mathbb{I}_p -\mathbb{I}_O,\; p\in\mathcal{C}$, where $\mathbb{I}_p$ is the indicator function at $p$. 
\end{prop}

Thus adding at most one finite point to $D$ makes it principal. A candidate for the denominator of a rational solution is then a polynomial of minimal degree vanishing on $D$, and it will contain at most one additional unnecessary root. We conclude this section with the algorithm for computing rational solutions.\\

\noindent
\textsf{RationalSolutions} Input: An equation of the form \eqref{eqell}\\
Output: A basis of rational solutions of \eqref{eqell}\\
\begin{enumerate}
\item Compute $D=$\textsf{UniversalDivisor} of \eqref{eqell}.
\item Compute $Q$ a minimal degree polynomial (i.e. a polynomial with minimal number of roots) vanishing on the affine part of $D$.
\item Note $P$ a generic polynomial of degree $\deg Q+D(O)$, and solve the linear system with the rational fraction $P/Q$
\item Return a basis of solutions. 
\end{enumerate}

\begin{proof}[Proof of Theorem \ref{thm1} for rational solutions]
If $f$ is a rational solution, its divisor is $\hbox{div}(f)\geq -D$ where $D$ is the universal divisor. For finite poles, we consider $\tilde{D}$ the affine part of $D$. The polynomial $Q$ computed in step $2$ then vanishes on $D$ at required orders at finite places and thus $fQ$ has no finite poles. The degree of $f$ is bounded by $D(O)$, and thus the degree of $fQ$ is bounded by $\deg Q+D(O)$. Thus step $4$ returns a basis of rational solutions of \eqref{eqell}.
\end{proof}

\begin{cor}
The basis of solutions returned has coefficients in $\mathbb{K}$.
\end{cor}

\begin{proof}
The universal divisor is defined with maximal ideals in $\mathbb{K}[u,v]$, and thus the polynomial $Q$ can be chosen with coefficients in $\mathbb{K}$. The linear system solving in step $3$ can also be done in $\mathbb{K}$, and thus the returned basis of solutions has coefficients in $\mathbb{K}$.
\end{proof}

\noindent
\textbf{Example 1}\\
$$(5403\sqrt{2} + 238 + (522\sqrt{2} - 2387)u - (1392\sqrt{2} + 48)v - (247 - 192\sqrt{2})u^3 +$$
$$(27\sqrt{2} + 348)u^2 + (-144\sqrt{2} + 560)uv)f((u,v)\oplus 2.\delta)$$
$$+(6117\sqrt{2} - 4078 - (522\sqrt{2} - 2387)u - (1536\sqrt{2} - 1024)v + 183u^3 + (549\sqrt{2} - 540)u^2 - 512uv) f(u,v)=0$$
The singularities are
$$\left[\left[\frac{2365\sqrt{2}+2418}{529} , -\frac{142211\sqrt{2}+238865}{12167}\right], -1\right], \left[\left[\frac{225361}{33489},\frac{109585196}{6128487}\right], 1\right],\left[\left[1, 4\right], -1\right], $$
$$\left[\left[\frac{17684189\sqrt{2}+ 25350354}{305809}, \frac{126593080067\sqrt{2}+178345046065}{-169112377}\right], -1\right],$$ 
$$\left[\left[2-3\sqrt{2},5\sqrt{2}-9\right], 1\right], \left[\left[2-3\sqrt{2}, 9 - 5\sqrt{2}\right], 1\right]$$
The universal divisor is
$$\left[\left[-3\sqrt{2} + 2, 9 - 5\sqrt{2}\right], 1\right], \left[\left[-3\sqrt{2} + 2, -9 + 5\sqrt{2}\right], 1\right]$$
One rational solution is found
$$f(u,v)=\frac{u-1}{u- 2+3\sqrt{2}}$$

\noindent
\textbf{Example 2}\\
$$(683755344u^5+428826069u^4-2283387894u^3v+23241765780u^3-4867086366u^2v+18906489126u^2-$$
$$3183541218uv+12311643924u-112099291146v+434160746253)f((u,v)\oplus 2.\delta)+$$
$$(710713440u^5+445733190u^4-2373413940u^3v+24158107800u^3-5058978660u^2v+19651906260u^2-$$
$$3309057180uv+12797049240u-116518976460v+451278195030)f((u,v)\oplus \delta)+$$
$$(1523907765u^4-1690718238u^3v+21340525860u^3-1438821414u^2v+5276955222u^2-$$
$$10483308378uv+40783964436u-113647578210v+440116351677)f(u,v)=0$$
The singularities are
$$\left[\left[\frac{1}{4}, \frac{31}{8}\right], -1\right], \left[\left[-\frac{119}{64}, \frac{1499}{512}\right], -1\right], \left[\left[\frac{225361}{33489}, \frac{109585196}{6128487}\right], -1\right], \left[\left[\right], 1\right],$$
$$\left[\left[\frac{-23124632487283 + 160099089\alpha}{8160589921352},\frac{71442233397165191733+833373502407169\alpha}{16484179465793084848}\right], 1\right],$$
$$\left[\left[\frac{- 23124632487283 -160099089\alpha}{8160589921352},\frac{71442233397165191733- 833373502407169\alpha}{16484179465793084848}\right], 1\right]$$
where $\alpha^2=-4553557895$. One rational solution is found
$$f(u,v)=\frac{u - 6v + 23}{u - 1}$$

Remark that building simple examples having rational solution is not immediate: even the simplest rational solutions as above have $3$ or $4$ singularities, which after two shifts give $10$ or more singularities, and some new ones possibly in field extensions.

\section{Pseudo-rational solutions}

Even for the classical shift on $\mathbb{C}$, hyperexponential functions with finitely many poles and roots are not always rational as a factor $\lambda^n$ can appear. Similarly, a multiplicative factor can appear in front a rational function for elliptic difference equations, by considering a non zero solution of equation $f(p\oplus \delta)-\lambda f(p)$. However, this is not enough to encompass all such pseudo-rational solutions.

\begin{defi}
We note $\Theta_{\lambda,\alpha,\beta}$ with $(\alpha,\beta)\in\mathcal{C},\; \lambda\in\mathbb{C}^*$ a solution of the equation
$$ \Theta_{\lambda,\alpha,\beta}(p\oplus \delta)-R(p) \Theta_{\lambda,\alpha,\beta} =0$$
where $R$ is rational and defined by having for divisor $-\mathbb{I}_{(\alpha,\beta)\ominus\delta}+\mathbb{I}_{(\alpha,\beta)}+\mathbb{I}_{\ominus\delta}-\mathbb{I}_{O}$ and $\lambda$ is the quotient of dominant coefficients of $R$.
\end{defi}

\begin{prop}
The function $\Theta_{\lambda,\alpha,\beta}$ is well defined, has a unique simple pole at $(\alpha,\beta)$ and a unique simple root at $O$.
\end{prop}

\begin{proof}
Let us first remark that the divisor $-\mathbb{I}_{(\alpha,\beta)\ominus\delta}+\mathbb{I}_{(\alpha,\beta)}+\mathbb{I}_{\ominus\delta}-\mathbb{I}_{O}$ is principal as
$$\ominus((\alpha,\beta) \oplus (\ominus\delta)) \oplus (\alpha,\beta)\oplus (\ominus\delta)=(\ominus(\alpha,\beta))\oplus \delta \oplus (\alpha,\beta) \oplus (\ominus\delta)=O$$
and thus $R$ exists. Now following the orbit from $(\alpha,\beta)$, we encounter a pole at $(\alpha,\beta)\ominus\delta$ which creates a singularity at $(\alpha,\beta)$ in $\Theta_{\lambda,\alpha,\beta}$, but which is immediately compensated at the next point $(\alpha,\beta)\oplus\delta$ thanks to the root of $R$ at $(\alpha,\beta)$. The same occurs at infinity where $\ominus\delta$ is a root of $R$ which creates a root in $\Theta_{\lambda,\alpha,\beta}$ at $O$ and is then immediately compensated at the next point $\delta$ by a pole of $R$ at $O$. Thus $\Theta_{\lambda,\alpha,\beta}$ has a unique simple pole at $(\alpha,\beta)$ and a unique simple root at $O$.
\end{proof}

In the limit case $(\alpha,\beta)=O$, the $\Theta_{\lambda,O}$ satisfies the much simpler equation
$$ \Theta_{\lambda,O}(p\oplus \delta)-\lambda \Theta_{\lambda,O} =0.$$
When the point $(\alpha,\beta)\in\mathcal{C}$ is of torsion, the divisor of $\Theta_{\lambda,\alpha,\beta}$ will be $-\mathbb{I}_{(\alpha,\beta)}+\mathbb{I}_{O}$, which is a torsion divisor. Thus a multiple of this divisor can be realized by a rational function $f$ on $\mathcal{C}$. Thus dividing $\Theta_{\lambda,\alpha,\beta}$ be a suitable $n$ th root of $f$ will produce a function without any poles nor roots. As $(f(p\oplus \delta)/f(p))^{1/n}$ will be rational as its divisor is principal by construction, the ratio $\Theta_{\lambda,\alpha,\beta}/f$ will still be hyperexponential, and without any root or poles, and thus equal to a function $\Theta_{\tilde{\lambda},O}$. Now when $\tilde{\lambda}$ is a root of unity, this function will be algebraic (it will have in fact finitely many values), and thus $\Theta_{\lambda,\alpha,\beta}$ will be algebraic.\\

\noindent
\textbf{Example}\\
The function $\Theta_{\lambda,x,y}$ is given by the following relation
$$\frac{\Theta_{\lambda,x,y}((u,v)\oplus \delta)}{\Theta_{\lambda,x,y}(u,v)}=\frac{((u-1)y-(v+4)x+4u+v)\lambda(x-1)^2}{(y+4)(-8y+(u-1)x^2 -(2u+1)x+u-30)}$$
For example, specializing the point $(x,y)$ to $(0,\sqrt{15})$ which is a $3$-torsion point on $\mathcal{C}$, we find the relation
$$\Theta_{(1921+496\sqrt{15})^{1/3},0,\sqrt{15}}(u,v)=\frac{c}{(v-\sqrt{15})^{1/3}}.$$
where $c$ is an unspecified constant.

\begin{prop}
A pseudo-rational solution can be written $F(u,v) \Theta_{\lambda,\alpha,\beta}(u,v)$ where $F$ is rational on $\mathcal{C}$.
\end{prop}

\begin{proof}
A pseudo-rational solution $f$ is by definition a hyperexponential solution which has only finitely many singularities and roots. Thus we can attach to it a divisor $D$. However, it could be not principal and a priori not proper. Now using a product of function $\Theta_{\lambda,\alpha,\beta}$, we can realize a divisor equal to $D$ for all finite points. It can still differ at $O$. Making the quotient between our pseudo-rational function $f$ and this product gives a hyperexponential function $g$ which has no poles nor roots except possibly at $O$. Thus $g(p\oplus \delta)/g(p)$ is a rational function which has a root at $\ominus\delta$ and a pole at $O$ with same multiplicity. Its divisor is thus $N\mathbb{I}_{\ominus\delta}+N\mathbb{I}_O$ for some $N\in\mathbb{Z}$. However, if this divisor was principal, then we would have $N.\delta=O$ which is forbidden as $\delta$ is not of torsion.

Thus $D$ can be realized as a product quotients of functions $\Theta_{\lambda,\alpha,\beta}$, and $D$ is proper as all divisors of the $\Theta_{\lambda,\alpha,\beta}$ are proper. Now a proper divisor can be reduced modulo the principal divisors to a divisor with a single finite point $p$ and $O$. Thus dividing $f$ by a suitable rational function will give a pseudo-rational function with only one pole at $p$ and one root at $O$. This is the divisor of $\Theta_{\lambda,p}$.

Thus the divisor $D$ can be realized by a function of the form $F(u,v) \Theta_{\lambda,\alpha,\beta}(u,v)$ where $F$ is rational on $\mathcal{C}$. This function is hyperexponential by construction, and thus dividing $f$ by it produces a hyperexponential function $g$ without any poles nor roots. Thus the quotient $g(p\oplus \delta)/g(p)$ is a rational function, and can neither have roots nor poles. It is thus a constant function. This constant factor can then be removed by adjusting the parameter $\lambda$ in $\Theta_{\lambda,\alpha,\beta}(u,v)$ which multiplies $R$ by a constant. Thus we can assume that $g(p\oplus \delta)=g(p)$ and thus $g$ is constant.
\end{proof}

Remark that we do not try to understand possible analytic properties of these pseudo-rational functions. Only the knowledge along each orbit is enough to determine if they are solution of an elliptic difference equation. Thus a constant function is a function which is constant on all orbits, and is from the point of view of the elliptic difference equation not distinguishable from the constant function everywhere. A similar phenomenon occurs for the classical shift in $\mathbb{C}$ where functions $n \rightarrow e^{2i\pi kn}$ with $k\in\mathbb{Z}$ are for all purpose constant as they are constant on all orbits.

We can nonetheless find possible analytic formulas for $\Theta$ as an exponential of a sum of an elliptic integral of the first kind (for the $\lambda$ parameter) and an elliptic integral of the third kind (for the $\alpha,\beta$ parameter). As already known cite, elliptic integrals of the third kind become elementary for a torsion singular point with logarithmic singularity, and thus taking the exponential we recover the algebraic nature of $\Theta$ function in this case.

We now want to design an algorithm to compute pseudo-rational solutions. The possible parameters $\lambda$ can be obtained by considering a point $p$ on a non singular orbit and computing
\begin{equation}\label{eqprod}
\lambda=\exp\left(\underset{n\rightarrow \infty}{\overline{\lim}} \frac{\ln f(p\oplus n.\delta)-\ln f(p)}{n}\right)
\end{equation}
where $f$ is a non zero solution of equation \eqref{eqell}. The $\lambda$ depends on the choice of $f$, so this limit has to be computed for all solutions. This expression gives the possible $\lambda$'s in a pseudo-rational solution as the rational factor in $f$ will induce in $\ln f(p\oplus n.\delta)$ a logarithmic growth, and the $\Theta$ function will for large $n$ have the behaviour of multiplying by $\lambda$ at each shift of $\delta$ on the orbit. We use $\overline{\lim}$ as it is unclear such quantity would converge for non pseudo-rational solutions.

Now expression \eqref{eqprod} is nice but totally unusable for explicit computation. Indeed we will most likely obtain transcendental numbers, and it is even unclear that for all choice of $f$ (a $l$ dimensional vector space) we will obtain a finite set of $\lambda$. In the shift case on $\mathbb{C}$, the possible $\lambda$'s can be obtained through local analysis at infinity, which is a fixed point for the shift. In particular, the $\lambda$ is always algebraic even for non hyperexponential solutions.

In the elliptic case, the solution to this problem is just to not solve it. Indeed, it would be convenient to know $\lambda$ in advance, but it is not necessary. Just replacing $f(p\oplus i.\delta) \rightarrow \lambda^i f(p\oplus i.\delta)$ in equation \eqref{eqell} will not change the singularities and thus neither the output of \underline{UniversalDivisor}. Thus in \textsf{RationalSolutions}, the linear system to solve will simply have $\lambda$ as a parameter appearing polynomially in the matrix of the equation.

The other difficulty is the function $\Theta_{\lambda,\alpha,\beta}$. The parameter $(\alpha,\beta)\in \mathcal{C}$ can be a priori anywhere on $\mathcal{C}$ as it is related to the ``default'' of principality of a divisor. We already know the locus of the poles, but only a bound on their multiplicities, and for the roots nothing apart a bound on their number. However
the reasoning for the construction of the universal divisor stays valid. Indeed, we only use the fact that the solution should have finitely many roots and poles, and thus

\begin{prop}
The output divisor $D$ of \textsf{UniversalDivisor} is such that for any pseudo-rational solution of equation \eqref{eqell}, its divisor $D'$ is such that $D'\geq -D$.
\end{prop}

\begin{proof}
The proof of Proposition \ref{propuniversal} never used the fact that $f$ is rational, only the fact that $f$ has finitely many roots and poles. Thus the proof applies for pseudo-rational solutions.
\end{proof}

Let us look at the algorithm \textsf{RationalSolutions}. After computing a universal denominator $Q$, the resulting divisor could be not principal, and thus for pseudo-rational solutions a function $\Theta$ would be needed. The solution is to relax by $1$ the constraint on the degree of the numerator $P$ and allow the compensation of the additional root with a function $\Theta$. As we cannot know this superfluous root before computing $P$, we need to make the computation of the kernel of the matrix while keeping the parameter $(\alpha,\beta)\in\mathcal{C}$.\\

\noindent
\textsf{PseudoRationalSolutions} Input: An equation of the form \eqref{eqell}\\
Output: A basis of pseudo-rational solutions of \eqref{eqell}\\
\begin{enumerate}
\item $i:=0, B_0:=\emptyset$, note $\hbox{Eq}$ equation \eqref{eqell}. While order of $\hbox{Eq}>0$ and $i=0 \hbox{ or} \sharp B_i >0$ do
\begin{enumerate}
\item Reduce the order of $\hbox{Eq}$ by removing the solution $B_i$, obtain a new equation $\hbox{Eq}$.
\item Compute $D=$\textsf{UniversalDivisor} of $\hbox{Eq}$ and compute $Q$ a minimal degree polynomial vanishing on the affine part of $D$.
\item Note $P$ a generic polynomial of degree $\deg Q+D(O)+1$ and replace $PQ^{-1}\Theta_{\lambda,\alpha,\beta}$ in equation \eqref{eqell}, divide the equation by $\Theta_{\lambda,\alpha,\beta}$ and use the formula for $\Theta_{\lambda,\alpha,\beta}(p\oplus \delta)/\Theta_{\lambda,\alpha,\beta}(p)$ to remove all $\Theta$ in the equation.
\item Solve the linear system given by the coefficients in $u,v$ of this equation, in function of the parameters $\lambda,(\alpha,\beta)\in \mathbb{C}^*\times \mathcal{C}$.
\item If non empty, note $B_{i+1}$ one solution. Increase $i$ by $1$.
\end{enumerate}
\item Return the list of the $B_i$
\end{enumerate}

\begin{proof}[Proof of Theorem \ref{thm1} for pseudo-rational solutions]
If $f$ is a pseudo-rational solution, its divisor is $\hbox{div}(f)\geq -D$ where $D$ is the universal divisor. For finite poles, the polynomial $Q$ computed in step $2b$ vanishes on the affine part of $D$, and thus $fQ$ has no finite poles. The number of roots of $fQ$ is now bounded by $\deg Q+D(O)$. However the roots of $f$ do not always form a principal divisor. We know however that, up to adding at most one point, the divisor can be made principal. We can thus relax the constraint on the degree by $1$, and we then look for a $P$ of degree $\leq \deg Q+D(O)+1$, which we multiply by a function $\Theta_{\lambda,\alpha,\beta}$ to remove this additional root. Thus $f$ can be written $P Q^{-1} \Theta_{\lambda,\alpha,\beta}$ for some $P,\lambda,(\alpha,\beta)$. Substituting this expression in equation \eqref{eqell}, simplifying it, and taking the coefficients in $u,v$ give a list of linear equations in the coefficients of $P$ whose solutions are solution of \eqref{eqell}. The parameters $\lambda,(\alpha,\beta)$ appear non linearly. In step $2d$, we look for all $\lambda,(\alpha,\beta)$ such that this system has a non zero solution. If it is non empty, we note $B_i$ one of them in step $2e$.

In the next loop, this solution will be removed in step $2a$ from equation $\hbox{Eq}$ and thus the order will reduce by one. If no solution is found, then the loop is stopped and there are no pseudo-rational solution in $\hbox{Eq}$. The loop always terminate as the order of $\hbox{Eq}$ decreases by one at each step. The returned list $B$ in step $3$ is a list of solution of equation \eqref{eqell} and they are linearly independent by construction.
\end{proof}

The while loop is necessary the ensure that all solutions are linearly independent. However, we can ask if continuum of solutions are possible in step $2d$. If a continuum exists, then by evaluation we obtain arbitrary many pseudo-rational solutions $f_1,\dots,f_k$ of \eqref{eqell}. Evaluating these solutions on a non singular orbit, we obtain the asymptotic estimates
$$ \sum c_i f_i(p\oplus n.\delta)=\sum  c_ie^{n\ln\lambda_i+o(n)}$$
Thus if a linear combination of these $f_i$ equals zero, so is the right expression. We see that the terms $e^{n\ln\lambda_i+o(n)}$ cannot compensate between each other, and thus the only possibility is that all $\lambda_i$ are equal. Thus if a continuum exists, $\lambda$ is fixed along it, and thus the parameter $(\alpha,\beta) \in\mathcal{C}$ is free.

Now sending the parameter $(\alpha,\beta)$ to $O$, and up to computing a series expansion for the rational part of the expression of the solution, we obtain a solution of the form $\frac{P}{Q} \Theta_{\lambda,O}$. These solutions are simpler than for a generic $(\alpha,\beta)$ and can be searched independently. Removing them first would ensure that only zero dimensional ideal in the parameter space are encountered, thus possibly reducing the computation time.\\

\noindent
\textbf{Example 1}\\
Consider the equation
$$21(17u^2-31u+46+9\sqrt{2}u- 4v\sqrt{2}+87\sqrt{2}-8v)(u-109)f((u,v)\oplus \delta)$$
$$-(1+3\sqrt{2})(21u-13+2v)(u-2+3\sqrt{2})(u-1)^2f(u,v)=0$$
The universal divisor is found
$$\left[\left[-3\sqrt{2} + 2, 9 - 5\sqrt{2}\right], 1\right], \left[\left[-3\sqrt{2} + 2, -9 + 5\sqrt{2}\right], 1\right], \left[\left[\frac{1}{4}, \frac{31}{8}\right], 1\right] $$
Now the solution is searched under the form
$$\frac{u^2c_2+uvb_1+uc_1+vb_0+c_0}{u^2+3\sqrt{2}u - \frac{9}{4}u-\frac{3}{4}\sqrt{2} + \frac{1}{2}} \Theta(\lambda, x, y, u, v)$$
and we find
$$\frac{u-\frac{1}{4}}{u^2+3\sqrt{2}u - \frac{9}{4}u-\frac{3}{4}\sqrt{2} + \frac{1}{2}} \Theta_{1, \frac{1}{4},\frac{31}{8}}(u, v) $$

\noindent
\textbf{Example 2}\\
Consider the equation
$$(168u^3 + 2268u^2v + 2580984u^2 - 4599uv - 5233662u + 192843v + 1890462)f((u,v)\oplus 2.\delta)$$
$$+(23808u^3 + 16u^2v - 2643418u^2 + 16uv + 5287148u - 191840v - 1900306)f((u,v)\oplus \delta)$$
$$-(23976u^3 + 2284u^2v - 62434u^2 - 4583uv + 53486u + 1003v - 9844)f(u,v)=0$$
The splitting field for singularities is $\mathbb{Q}(i\sqrt{3})$, and the universal divisor is $[[\tfrac{1}{4}, \tfrac{31}{8}], 1]$. Two independent pseudo rational solutions are found
$$1,\;  \Theta_{1, \frac{1}{4}, \frac{31}{8}} (u, v)$$

\section{Hyperexponential solutions}

\begin{prop}\label{prophyper}
A hyperexponential solution can be written
$$H(u,v)F(u,v) \Theta_{\lambda,\alpha,\beta}(u,v)$$
where $F(u,v)$ is rational, and $H(u,v)$ is hyperexponential such that $H(p\oplus \delta)/H(p)$ is a rational function whose every root and pole lie on different orbits by $\phi_\delta$.
\end{prop}

The function $H$ plays the role of a product of $\Gamma$ function for the classical shift on $\mathbb{C}$. The main difference is that it is a priori not possible to split each factor to have a product of function for each orbit because all divisors are not principal.

\begin{proof}
Consider $f$ a hyperexponential function and an orbit. Along the orbit, we can encounter several poles and roots of $f(p\oplus \delta)/f(p)$. We can remove a pole or a root by multiplying or dividing by a function $\Theta$. By doing this repeatedly, we can remove finitely many poles and roots, and thus ensure that the valuation along the orbit is constant everywhere except for a single jump. This process can be done for all orbits which have a root or pole of $f(p\oplus \delta)/f(p)$, except the orbit of $O$.

On the orbit of $O$, we can have either finitely many roots/poles, and then by a product of $\Theta$ send them all to $O$. Or there are infinitely many. In this case, we can multiply $f$ by a quotient $\Theta_{\lambda,n.\delta}/\Theta_{\lambda,\alpha,\beta}$ where $(\alpha,\beta)$ such that its orbit $T$ never encounter the orbits with infinitely many roots/poles. The $n.\delta$ is chosen to remove poles or roots on the orbit, such that after finitely many products quotient, the orbit of $O$ has a single jump in valuation. Now the poles and roots on $T$ can also be concentrated to a single point of this orbit by multiplication/division of $\Theta$ functions, a point we note $(\alpha,\beta)$.

We now call $g$ the resulting hyperexponential function. The set of jump points define a divisor $D$, which is not a priori principal nor proper. However, we know that $g(p\oplus \delta)/g(p)$ has a principal proper divisor as it is a rational function. This divisor equals $D-N.\mathbb{I}_{(\alpha,\beta)}+N.\mathbb{I}_{(\alpha,\beta)\ominus \delta}$ where $N$ is the valuation of $g$ at $(\alpha,\beta)$. As this divisor is proper and so is $-N.\mathbb{I}_{(\alpha,\beta)}+N.\mathbb{I}_{(\alpha,\beta)\ominus \delta}$, the divisor $D$ is thus also proper. We can now reduce $D$ modulo principal divisors to a divisor $\mathbb{I}_{q}-\mathbb{I}_{O}$. This reduction is equivalent to divide $g$ by a hyperexponential function $H$ where $H$ is such that $H(p\oplus \delta)/H(p)$ is a rational function whose every root and pole lie on different orbits by $\phi_\delta$. Now the resulting function $\tilde{g}$ has the following properties
\begin{itemize}
\item It is hyperexponential
\item The orbits of $q$ and $O$ are the only ones with infinitely many roots/poles.
\item There is at most two jumps in valuation along orbits, and of amplitude at most $1$.
\item The roots/poles outside the orbits of $q$ and $O$ can only be at $(\alpha,\beta)$.
\end{itemize}

So now two situations occur. Either $q$ and $O$ are on the same orbit. Then the two jumps compensate as the divisor of $\tilde{g}(p\oplus \delta)/\tilde{g}(p)$ has to be proper. And thus the orbit has finitely many poles or roots, and then $\tilde{g}$ is pseudo-rational.

Or $q$ and $O$ are not on the same orbit. Then we look at the divisor of $\tilde{g}(p\oplus \delta)/\tilde{g}(p)$, which writes
$$\mathbb{I}_{q}-\mathbb{I}_{O} -N.\mathbb{I}_{(\alpha,\beta)}+N.\mathbb{I}_{(\alpha,\beta)\ominus \delta} $$
It should be principal, but when performing the computation, we find
$$q\ominus N.(\alpha,\beta)\oplus N.(\alpha,\beta)\ominus N.\delta=q\ominus N.\delta=O$$
thus $q\ominus O=N.\delta$, which is a contradiction as we assumed they were not on the same orbit.

We thus conclude that $\tilde{g}$ is pseudo-rational, and going backwards, that our initial hyperexponential function $f$ can be written $H(u,v)F(u,v) \Theta_{\lambda,\alpha,\beta}(u,v)$
\end{proof}

Remark that in the construction of the function $H$, the specific position of the valuation jump on an orbit is not important. We can always shift these jumps by a multiple of $\delta$ (although this would change the pseudo-rational part $F(u,v) \Theta_{\lambda,\alpha,\beta}(u,v)$). Of course the divisor formed by the jumps is not always principal, and this suggests to define a more relaxed version of principality.

\begin{defi}
We say that a divisor $D$ is $\delta$ principal if the weighted sum of its points for elliptic addition law is a multiple of $\delta$.
\end{defi}

A principal divisor is $\delta$ principal, but so are any divisor built by shifting by $\delta$ arbitrary the points of its support, and thus is a property invariant by $\delta$ shifts.

\begin{prop}\label{propjump}
Consider a hyperexponential solution $f$, and all its orbits such that
$$m_p=\lim_{n\rightarrow +\infty} \hbox{val}_{p\oplus n.\delta} f- \lim_{n\rightarrow -\infty} \hbox{val}_{p\oplus n.\delta} f \neq 0.$$
Then the divisor formed by
$$D=\sum_{\mathcal{O} \in \hbox{Orbits}} m_p \mathbb{I}_p,\quad p\in \mathcal{O}$$
is proper and $\delta$-principal.
\end{prop}

\begin{proof}
We follow the proof of Proposition \ref{prophyper} and look at the quantities $m_p$. Changing finitely many roots/poles does not change the $m_p$. Thus the function $g$ as defined before has the same $m_p$. Now $g$ is such that an orbit contains either a single jump, either this is the orbit of $(\alpha,\beta)$ which contains this single point as root/pole. The amplitude of the jump is thus given by $m_p$. We now know that the divisor $D$ of the jumps of $g$ reduces modulo principal divisors to $\mathbb{I}_{q}-\mathbb{I}_{O}$, and that $q$ is on the same orbit as $O$, i.e. $q$ is a multiple of $\delta$. Thus $D$ is $\delta$ principal.

Let us now remark that shifting a jump by $\delta$ in $g$ shifts a point in $D$ by $\delta$. But this new $D$ is still $\delta$ principal. So the $\delta$ principality of
$$\sum_{\mathcal{O} \in \hbox{Orbits}} m_p \mathbb{I}_p,\quad p\in \mathcal{O}$$
does not depend on the choice of the point in an orbit $\mathcal{O}$. As one choice (the divisor $D$) is $\delta$-principal, so are any of these.
\end{proof}

In practice, we can compute bounds for the jump amplitude $m_p$ by counting singularities on orbits. Then for any combination of choice of $m_p$, we build a divisor, and it should be $\delta$-principal. We can test this property, and if true, we only have to translate by a multiple of $\delta$ one of its points to obtain a principal divisor. This will be a suitable choice for the divisor of $H(p\oplus \delta)/H(p)$ in Proposition \ref{prophyper}.

We will test $\delta$-principality for many divisors with the same support. To do it efficiently, we can first compute the regulator matrix of the points of the support $p_1,\dots, p_k$
$$\hbox{Reg}= (<p_i,p_j>)_{i,j=1\dots k}$$
where $<\;,\;>$ is the scalar product defined by the canonical height. We can now define a norm on divisors by
$$\left\Vert \sum n_i \mathbb{I}_{p_i} \right\Vert^2= \frac{1}{<\delta,\delta>} (n_1,\dots,n_k) \hbox{Reg}(n_1,\dots,n_k)^\intercal$$
and then $\Vert D \Vert$ gives the value (up to sign) of the number of shifts of $\delta$ to test the $\delta$ principality of $D$. By computing the matrix norm $\Vert \hbox{Reg} \Vert_2$, we obtain an absolute bound as we can bound $(n_1,\dots,n_k)$ as this is simply the number of singularities of equation \eqref{eqell}.\\

\noindent
\textsf{HyperexponentialSolutions} Input: An equation of the form \eqref{eqell}\\
Output: A basis of hyperexponential solutions of \eqref{eqell}\\
\begin{enumerate}
\item Compute $S=$\textsf{Singularties} of \eqref{eqell}.
\item For each pair of ideals of $S$, check if the field they define are the same. If yes, add the difference of two roots in a set $\Sigma$
\item Compute $b=\lceil\hat{h}(\delta)^{-1/2}\max_{p\in\Sigma} \hat{h}(p)^{1/2}\rceil$
\item For each element $\mathcal{J}$ of $S$ do
\begin{enumerate}
\item Compute $\phi_\delta^{-i}$ of a root of $\mathcal{J}$ for $i=1\dots b$, and check if one of them is a root of one of the ideals of $S$ other than $\mathcal{J}$
\item If no, compute $\phi_\delta^i$ of a root $p$ of $\mathcal{J}$ for $i=1\dots b$ and compute a list $L$ of the positive and minus the negative multiplicities for the ideals $\mathcal{J}_i$ vanishing on $\phi_\delta^i(p)$.
\item Compute the list $I_{\mathcal{J}_i}$ of integers from $-\sum_{i=1}^b L_{i,2}$ to $\sum_{i=1}^b L_{i,1}$.
\end{enumerate}
\item Note $L$ the list formed by roots of $\mathcal{J}$ with their associated list $I_{\mathcal{J}}$.
\item Compute $b=\left\lceil\Vert \hbox{Reg} \Vert_2 \sharp L /\sqrt{<\delta,\delta>} \right\rceil$ where $\hbox{Reg}$ is the regulator matrix of the points of $L$.
\item For each combination of integers in the second part of $L$ do
\begin{enumerate}
\item Build the divisor $D$ with the chosen multiplicities.
\item Test principality of $D$ and its shifts of one of its point by $N.\delta$ for $ \mid N \mid \leq b $.
\item If one of it is principal, note it $\tilde{D}$, find $R$ rational realizing this divisor, and change the unknown function $f$ in equation \eqref{eqell} by multipying it by a hyperexponential function $H$ defined by $R$. Note $\hbox{Eq}$ this new equation.
\item Compute pseudo-rational solutions of $\hbox{Eq}$, and multiply them by $H$
\end{enumerate}
\item Return the list of all solutions found
\end{enumerate}

The list $L$ contains all the roots of the ideals $\mathcal{J}_i$ not satisfying step $5a$. This requires to computed a splitting field containing all these roots, which can be very large. Remark that the regulator computation and the $\delta$ principality test is in principle not required, but it allows to throw up combinations without testing for pseudo-rational solutions. Indeed, even if $\tilde{D}$ is not principal, we can add a point to it to build $R$, and by Proposition \ref{prophyper}, the additional singularity (if this is really a combination for a hyperexponential solution) should be on the same orbit as $O$, and then will still reduce to a pseudo-rational solution after division by $H$.

\begin{proof}[Proof of Theorem \ref{thm1} for hyperexponential solutions]
If $f$ is a hyperexponential solution, with Propositions \ref{prophyper}, \ref{propjump}, we know that we can write
$$f(u,v)=H(u,v)F(u,v) \Theta_{\lambda,\alpha,\beta}(u,v)$$
where $H$ is hyperexponential with single valuation jumps on orbits. Steps $1,2,3$ compute the roots of $a_0,a_l$ and dispersion bound $b$. Then for each maximal ideal representing a root, we check it is the first on its orbit (step $4a$), and if yes compute all singularities encountered along the orbit (step $4b$). Then all roots multiplicities of $a_l$ encountered are added, and all roots multiplicities of $a_0$ encountered are added, and a list $I_{\mathcal{J}_i}$ of all integers between them is built (step $4c$). These integers are all the possible valuation change along the orbit, and as no further change can occur before or after the region studied (because of the bound $b$ on the dispersion), we know that the quantity
$$m_p=\lim_{n\rightarrow +\infty} \hbox{val}_{p\oplus n.\delta} f- \lim_{n\rightarrow -\infty} \hbox{val}_{p\oplus n.\delta} f$$
is among them.

In step $5$, the list of roots is build by computing a common splitting field for all the ideals $\mathcal{J}$. So now the list $L$ contains all possible points (which are all on different orbits) for which $H$ can admit infinitely many roots/poles, and the lists of integers contain all the possible valuation jump on these orbits. In step $7$, we consider one by one every possible combination. Step $7a$ build the divisor, step $7b$ checks its $\delta$-principality (using the regulator computed in step $6$), and if yes, step $7c$ finds a divisor $\tilde{D}$ which equals $D$ except for a suitable shifted point, which makes it principal. Then a rational function $R$ is computed, which is the candidate for $H(p\oplus \delta)/H(p)$, and equation \eqref{eqell} is transformed accordingly. In step $7d$, if the multiplicities were well chosen (every possible combination has to be tested), then $f(p)/H(p)$ will have finitely many roots/poles, and thus will be pseudo-rational. Then algorithm \textsf{PseudoRationalSolutions} will find it. Step $8$ returns all solutions found.

Let us check that the output is indeed a basis. For each valuation jump combination, the solutions are linearly independent as those in the output of \textsf{PseudoRationalSolutions} are. If we consider two solutions with different valuation jumps, then along an orbit with different jump values, the poles/roots far enough on the orbit would differ, and thus no linear relation is possible. Thus all solutions returned are linearly independent, and as all hyperexponential are found, this indeed forms a basis.
\end{proof}

Complexity of this algorithm is in theory horrible. The splitting field can be very large, the combinatorics part also, and only these two combined cost at worst $2^d d!$ where $d$ is the number of roots of $a_0a_l$. Remark that even for the classical shift in $\mathbb{C}$, the same difficulties occur cite, and only practical improvements have been made cite. In practice, for small examples with degrees $\leq 4$, this is still manageable.\\

\noindent
\textbf{Example 1}\\
$$uf((u,v)\oplus \delta)-(u-1) f(u,v)=0$$
This equation admits by construction a hyperexponential solution, as it is of order $1$. Let us look nonetheless how the algorithm works. The splitting field is $\mathbb{Q}(\sqrt{15})$ and the singularities found are
$$\left[\left[0, \sqrt{15}\right], -1\right], \left[\left[0, -\sqrt{15}\right], -1\right], \left[\left[1, 4\right], 1\right], \left[\left[1, -4\right], 1\right].$$
This singularity set contains $3$ orbits
$$\mathcal{O}_1: \hbox{first point } (0, \sqrt{15}), \hbox{ valuation jump } \in \{-1,0\}$$
$$\mathcal{O}_2: \hbox{first point } (0,-\sqrt{15}), \hbox{ valuation jump } \in \{-1,0\}$$
$$\mathcal{O}_1: \hbox{first point } (1,-4), \hbox{ valuation jump } \in \{0,1,2\}$$
Thus $2\times 2\times 3$ combinations have to be tested. For each one $\delta$ principality is tested, and two of them can be realized
$$[],\;\; H_1=1$$
$$\left[\left[\frac{122880-23984\sqrt{15}}{14161}, \frac{46049280-14043481\sqrt{15}}{1685159}\right], -1\right], \left[\left[0, -\sqrt{15}\right], -1\right], \left[\left[1, -4\right], 2\right]$$
$$H_2=119(8\sqrt{15}+29)\frac{24u\sqrt{15}+64v\sqrt{15}-119u^2+232\sqrt{15}+151u-232v-960}{8(14161u-122880+23984\sqrt{15})u)}$$
Remark that there is not $(u-1)/u$. Its divisor is equivalent up to shift in $\delta$ to the divisor of $H_2$. We see that the point $(0, \sqrt{15})$ has been shifted two times. For obtaining $(u-1)/u$, it would have been necessary to shift $(1, -4)$ instead. But this is a perfectly possible choice, as a rational function and $\Theta$ function will compensate for this. A hyperexponential solution is indeed found with the choice $H_2$
$$\hbox{Hyper}\left(119(8\sqrt{15}+29)\frac{24u\sqrt{15}+64v\sqrt{15}-119u^2+232\sqrt{15}+151u-232v-960}{8(14161u-122880+23984\sqrt{15})u)}\right)$$
$$\Theta_{\frac{226432\sqrt{15}-878888}{14161}, \frac{122880+23984\sqrt{15}}{14161},\frac{-14043481\sqrt{15} - 46049280}{1685159}}(u,v)$$
$$\frac{464u\sqrt{15}+55v\sqrt{15}-119u^2-244\sqrt{15}+1920u+244v - 825}{u - 1}$$\\

\noindent
\textbf{Example 2}\\
$$(3u-8v+29)(u-1)^2 f((u,v)\oplus 2.\delta)+(u-6v+23)(141u^2-10uv-98u-374v+1493)f(u,v)=0$$
The singularities are
$$\left[\left[-\frac{119}{64}, \frac{1499}{512}\right], -1\right], \left[\left[1, -4\right], -2\right], \left[\left[109, 1138\right], 1\right], \left[\left[\frac{1}{4}, \frac{31}{8}\right], 1\right],$$
$$ \left[\left[-\frac{11}{9}, \frac{98}{27}\right], 1\right], \left[\left[\frac{1201}{100}, \frac{41801}{1000}\right], 1\right], \left[[], -1\right]$$
This singularity set contains $3$ orbits
$$\mathcal{O}_1: \hbox{first point } \left(-\frac{119}{64}, \frac{1499}{512}\right), \hbox{ valuation jump } \in \{-4,-3,-2,-1,0\}$$
$$\mathcal{O}_2: \hbox{first point } (109, 1138), \hbox{ valuation jump } \in \{0,1,2\}$$
$$\mathcal{O}_1: \hbox{first point } \left(\frac{1201}{100}, \frac{41801}{1000}\right), \hbox{ valuation jump } \in \{0,1,2\}$$
Thus $5\times 3\times 3$ combinations have to be tested. For each one $\delta$ principality is tested, and three of them can be realized
$$1,\; \frac{-2560u^2-87328u+35136v-50656}{35136u^2+30195u-65331},$$
$$\frac{179956361-1386240u^3+51200u^2v-45603220u^2+12713312uv-46903105u+63907820v}{\tfrac{25}{128}(64u+119)^3}$$
For the second one, two pseudo rational solution are then found
$$\hbox{Hyper}\left(\frac{-2560u^2-87328u+35136v-50656}{35136u^2+30195u-65331} \right)\Theta_{\frac{549}{32},O}(u, v) \frac{u-1}{-\frac{u}{6} + v - \frac{23}{6}},$$
$$\hbox{Hyper}\left(\frac{-2560u^2-87328u+35136v-50656}{35136u^2+30195u-65331} \right)\Theta_{-\frac{549}{32},O}(u, v)\frac{u-1}{-\frac{u}{6} + v - \frac{23}{6}}$$

\section{Conclusion}

The algorithms are implemented in Maple, except for the computation of the canonical height done in PARI/GP. The Maple implementation is done by assuming that $\mathbb{K}$ contains all the roots of $a_la_0$, which is always possible to assume by considering their splitting field. This is normally only necessary for the computation of hyperexponential solutions, but the implementation is easier as some properties are trivialized (as Proposition \ref{propideals}), and the code works in practice only on small examples anyway. One of the difficulties is that shifting an integer by $1$ hardly grows its bit size (logarithmic growth), contrary to the shift on an elliptic curve for which the size growth is quadratic (except for a $\delta$ of torsion, which is forbidden). Thus rapidly coefficient size explodes, and even simple examples are difficult. Also, the $\Theta$ function parameters appear non linearly in the equation, and it would be interesting to know if, using the a priori knowledge on the solutions (a continuum in the parameter space is impossible), we could have a better approach than Groebner basis computation we used. These $\Theta$ functions can be algebraic, which is a novelty in comparison to the classical shift on $\mathbb{C}$. As in the classical shift on $\mathbb{C}$, ramification points for solutions are forbidden, and these functions indeed always have a divisor with integer coefficients, but they have a finite global monodromy which is possible as a complex elliptic curve is a torus and so has non trivial homology. Finally, the $\delta$-principality test is intriguing. This does not appear in the classical shift on $\mathbb{C}$, as then every combination of exponents has to be tested. Here, a purely algebraic property on divisors (thus independent on the difference equation) is required for a combination to be a priori possible. The examples suggest that it is rarely satisfied even for examples carefully tuned. Thus the combinatorial part could possibly be avoided by some kind of LLL algorithm using the addition law on elliptic curves.

\label{}
\bibliographystyle{plain}


\end{document}